\documentclass[12pt]{amsart}
\usepackage{calc,pifont}
\usepackage{graphicx,adjustbox}
\usepackage{mathrsfs}
\usepackage{amssymb,amsbsy}
\usepackage{amscd}
\usepackage{enumerate}
\usepackage{setspace}
\usepackage{tikz}
\usepackage{tikz-cd}
\usepackage{bm}
\usepackage{cite}
\usepackage{graphicx}
\usepackage{hyperref}
\allowdisplaybreaks
\usepackage{cite}

\usepackage{kpfonts}
\usepackage{stackengine}
\usepackage{calc}
\newlength\shlength
\newcommand\xshlongvec[2][0]{\setlength\shlength{#1pt}%
  \stackengine{-5.6pt}{$#2$}{\smash{$\kern\shlength%
    \stackengine{7.55pt}{$\mathchar"017E$}%
      {\rule{\widthof{$#2$}}{.57pt}\kern.4pt}{O}{r}{F}{F}{L}\kern-\shlength$}}%
      {O}{c}{F}{T}{S}}
  
\usepackage[letterpaper]{geometry}
\theoremstyle{definition}
\newtheorem{theorem}{Theorem}[section]
\newtheorem{lemma}[theorem]{Lemma}
\newtheorem{proposition}[theorem]{Proposition}
\newtheorem{definition}[theorem]{Definition}
\newtheorem{corollary}[theorem]{Corollary}

\newtheorem{example}[theorem]{Example}

\newcommand{\Q}{{\mathbb{Q}}}

\newcommand{\N}{{\mathbb{N}}}

\title{Generalized cell structures}

\author{Ana G.~Hern\'andez-D\'avila}
\author{Benjam\'in A.~Itz\'a-Ortiz} 

\address{Centro de Investigaci\'on en Matem\'aticas, Universidad Aut\'onoma del Estado de Hidalgo, Pachuca, Hidalgo, Mexico}
\email{itza@uaeh.edu.mx,anagabriela169@hotmail.com}

\author{Roc\'io Leonel-G\'omez}
\address{Facultad de Ciencias, Universidad Nacional Aut\'onoma de M\'exico, Mexico City, Mexico}
\email{rocioleonel@gmail.com}

\begin{document}

\begin{abstract} 
 Cell structures were introduced by  W.~Debski and E.~Tymchatyn as a way to study some classes of topological spaces and their continuous functions by means of discrete approximations. In this work we weaken  the notion of cell structure and prove that the resulting class of topological space admitting such a generalized cell structure  includes non-regular spaces.
\end{abstract}
\maketitle

\section*{Introduction}

Cell structures may be thought of as  devices to represent some topological spaces  by means of discrete approximations, more precisely, to describe a space $X$ as homeomorphic to a perfect image of an inverse limit of graphs. They were introduced in \cite{DT} for complete metric spaces and extended in \cite{DT2018} to topologically complete spaces.  The initial step in identifying a cell structure consists in considering an inverse sequence of  graphs, each of which have attached a  reflexive and symmetric relation.
The vertices of each graph are the elements of a  set $G$ and the edges are described by reflexive and symmetric subsets of $G\times G$, or equivalently, the set of edges of $G$ is an entourage of the diagonal of $G$. 
Such an inverse sequence is said to be a cell structure when it admits a couple of properties: one that allows the induced natural relation on its inverse limit to be an equivalence relation (where two threads in the inverse limit are declared to be related if  every pair of their corresponding components are related), and a second property that allows the quotient of the inverse limit by its natural relation to be a perfect mapping. This resulting quotient space is the said to admit, up to homeomorphism, a cell structure.

In this paper we explore the natural question on the  class of topological space obtained by admitting a weaker version of a cell structures, which we will call generalized cell structures, or g-cell structure for short. We are able to show that there are topological spaces admitting such a g-cell structure but not a cell structure. Furthermore, an explicit  example of a nonregular space admitting a g-cell structure will be provided.

Representing the structure of spaces as approximations of simpler more understandable structures,  such as graphs, is a fruitful idea used not only in topology \cite{L}, but in a variety of other areas such as spin networks  \cite{nash:1999}, operator algebras \cite{elliott}, networks \cite{newman}, among others.

We divide this work in two sections. In Section~1 we introduce the notion of generalized cell structure and prove the basic results needed in the rest of the paper. The main results are given in Section~2.

The authors gratefully acknowledge that this paper have benefited from stimulating conversations with Carlos Islas and Juan Manuel Burgos. The first author received support for this work from CONACyT scholarship Num.~926215.

\section{Preliminaries}

In this section we give the definition of spaces admitting a g-cell structure together with some preliminary results which will be useful in the next section. For any set $G$, we will denote the diagonal of $G$ as $\Delta_G=\{(x,x)\colon x\in G\}$. When no confusion arises, we will write $\Delta$ rather than $\Delta_G$. The set of natural numbers is as usual $\N=\{ 1,2,3,\ldots\}$.

\begin{definition}
We say that an order pair $(G,r)$ is a cellular graph if $G$ is a non\-empty topological space and $r\subset G\times G$ is a reflexive and symmetric relation on $G$. The vertices of the graph, the elements of $G$,  will also be known as cells of the graph, while  the elements of $r$ are the edges of the graph. 
\end{definition}


\begin{definition}
  If $(G,r)$ is a cellular graph and $u\in G$ is a cell, we define the neighborhood of $u$ to be  the set
 \begin{equation*}
     B(u,r)=\{v\in G\colon (u,v)\in r\}.
 \end{equation*}
 Therefore, the set $B(u,r)$ is the set of vertices of $G$ that are adjacent to $u$ in $G$.
 More generally, for $A\subset G$ we denote
 \begin{equation*}
     B(A,r)=\bigcup_{a\in A} B(a,r).
 \end{equation*}
 \end{definition}

Following the terminology of \cite[Section~8.1, pg.~426]{E}, given a cellular graph $(G,r)$, since the relation $r$ is reflexive, then the diagonal $\Delta$ of $G$ is a subset of $r$. On the other hand, since $r$ is symmetric it follows that $r$ is equal to its own inverse relation $-r$, where $-r=\{(x,y)\colon (y,x)\in r\}$. 
Conversely, given a  entourage of the diagonal $r$ of $G$, that is to say, a relation on $G$ which contains $\Delta$ and satisfies $r=-r$, it follows that $r$ is a reflexive and symmetric relation. Thus, we may characterize a cellular graph as a pair $(G,r)$ where $G$ is a nonempty topological space and $r$ is an entourage of the diagonal.
Finally, the composition of $r+r=2r$ of the relation $r$ with itself is defined as $2r=\{(x,z)\colon\exists y\in G,\,\, (x,y),(y,z)\in r  \}$. In other words a pair of cells $(x,z)$ belongs to $2r$ if there exits a path of length 2 in the cellular graph $(G,r)$ joining them. Notice that $2r$ is still a reflexive and symmetric relation of $G$. Furthermore, it is always true that $B(x,r)\subset B(x,2r)$ and the opposite inclusion follows if  and only if $r$ is transitive.

\begin{definition}
  Suppose that $\{(G_n,r_n)\}$ is a sequence of cellular graphs and let $\{g_n^{n+1} \}_{n\in\N}$ be a family of continuous functions, called bonding maps,  $g_n^{n+1}\colon G_{n+1}\to G_n$ such that
  \begin{itemize}
  \item $g_n^n$ is the identity on $G_n$,
  \item $g_n^l=g_n^m\circ g_m^l$ for $n<m<l$ and
  \item the bonding maps send edges to edges, that is, $(g_n^{n+1}(x),g_n^{n+1}(y))\in r_{n}$ whenever $(x,y)\in r_{n+1}$.
  \end{itemize}
  We will then say that $\left\{\left(G_n,r_n\right),g_{n}^{n+1}\right\}$ is an inverse sequence of cellular graphs.
\end{definition}

  The inverse limit of an inverse sequence of cellular graphs $\left\{\left(G_n,r_n\right),g_{n}^{n+1}\right\}$ will be denoted by $G\sb\infty$ or $\displaystyle\lim_{\longleftarrow} \{G_i,g_{i}^{i+1}\}$. It is a nonempty subspace of $\prod_{n\in\N}G_n$ with the product topology.  Elements in $G\sb\infty$ are called threads and the mappings $g_i\colon G\sb\infty\to G_i$ denote the restriction of the standard projection maps $p_i\colon \prod_{n\in\N} G_n\to G_i$.  In other words,
  
 \begin{equation*}
     G\sb\infty=\left\{\bar{x}=(x_n)_{n\in\N}\colon \forall j\geq i, \,\,
                      g_i^j(x_j)=x_i \right\}\subset\prod_{n\in\N}G_n.
 \end{equation*}

The topology in $G\sb\infty$ is actually characterized by a basis, as the following result shows.
It is actually a special  case of \cite[Propositon~2.5.5]{E}. We include a proof for completeness.

\begin{proposition}\label{basis}
  Let $G\sb\infty$ be the inverse limit of an inverse sequence of cellular graphs $\left\{\left(G_n,r_n\right),g_{n}^{n+1}\right\}$. Then the collection of the sets of the form 
  \begin{equation*} \left\{ g_{i}^{-1}\left(A_{i}\right) \colon
  i\in\N \text{ and }
  A_{i}\subset G_{i} \text{ is an open in the base of } G_i 
  \right\}
  \end{equation*}
  defines a basis for the the topology of $G\sb\infty$.
\end{proposition}
\begin{proof}
Let $\bar{x}\in G_{\infty}$ and let $U\subset G_{\infty}$ be an open set such that $\bar{x}\in U$. By definition of the subspace topology, there exists an open set $V\subset \prod_{n\in\N} G_{i}$ such that $U=V\cap G_{\infty}$. Thus there exists an open set $V_{0}=p_{j_{1}}^{-1}(U_{1})\cap \cdots \cap p_{j_{n}}^{-1}(U_{n})$ in the base of $\prod_{n\in\N} G_{i}$ satisfying $\bar{x}\in V_{0}\subset V$. Let $j=\max\{j_{i}\}$. Since the bounding maps are continuous, for every $i\in \{1,...,n\}$, $\bigl(g_{j_{i}}^{j}\bigr)^{-1}(U_{i})$ is an open set in $G_{j}$. Hence, $\bigcap \bigl(g_{j_{i}}^{j}\bigr)^{-1}(U_{i})$ is also an open set and $\bar{x}\in \bigcap \bigl(g_{j_{i}}^{j}\bigr)^{-1}(U_{i})$. Let $V_{j}$ be an open set in the base of $G_{j}$ such that $x_{j}\in V_{j}\subset \bigcap \bigl(g_{j_{i}}^{j}\bigr)^{-1}(U_{i})$. Now, since $g_{j}^{-1}\bigl(\bigl(g_{j_{i}}^{j}\bigr)^{-1}(U_{i})\bigr)=g_{j_{i}}^{-1}(U_{i})=G_{\infty}\cap p_{j_{i}}^{-1}(U_{i})$, we get that $\bar{x}\in g_{j}^{-1}(V_{j})\subset G_{\infty}\cap p_{j_{i}}^{-1}(U_{i})\subset G_{\infty}\cap V=U$.
\end{proof}

\begin{definition}\label{natrel}
Let $\left\{\left(G_n,r_n\right),g_{n}^{n+1}\right\}$ be an inverse sequence of cellular graphs and let $G\sb\infty$ be its inverse limit. We define the  natural relation $r$ on $G\sb\infty$ by 
\begin{equation*}
         r=\{(\bar{x},\bar{y})\in G\sb\infty\times G\sb\infty
             \colon\forall n,\,\, (x_n,y_n)\in r_n\}.
    \end{equation*}
\end{definition}

In turns out that the inverse limits of an inverse sequence of cellular graphs is actually a cellular graph with the natural relation $r$ on $G\sb\infty$ given in Definition~\ref{natrel}, as established in the following proposition.

\begin{proposition}\label{relation}
     If $G\sb\infty$ is the inverse limit of an inverse sequence of cellular graphs then 
    it is a cellular graph with respect its natural relation given in Definition~\ref{natrel}. 
\end{proposition}
\begin{proof}
    Let $\left\{\left(G_n,r_n\right),g_{n}^{n+1}\right\}$ be an inverse sequence of cellular graphs and let $G\sb\infty$ be its inverse limit.
    By definition, $G\sb\infty$ is a topological space. We only need to check that  $r$ is reflexive and symmetric. This follows immediately since each $r_n$ is reflexive and symmetric.
\end{proof}



We are now ready to define a generalized cell structure. It  consists of just one condition which is a necessary condition implied by  the first  original  condition  for cell structures originally defined in \cite{DT}. It enables the natural relation induced  on the inverse limit of an    inverse sequence of cellular graphs to be transitive and therefore an equivalent relation. 

\begin{definition}\label{GCellStruc}
We say that an inverse sequence of cellular graphs   $\left\{\left(G_n,r_n\right),g_{n}^{n+1}\right\}$ is a generalized cell structure, or a g-cell structure for short, if its natural relation $r$ given in Definition~\ref{natrel} is an equivalence relation.
\end{definition}

\begin{proposition}\label{Prop:implyGcell}
Let $\left\{\left(G_n,r_n\right),g_{n}^{n+1}\right\}$ be an inverse sequence of cellular graphs and let $G\sb\infty$ be its inverse limit.
Each of the following statement implies the next.
\begin{enumerate}
    \item For each $i\in\N$, the relation $r_i$ is an equivalence relation on $G_i$.
    \item  For each thread $\bar{x}\in G\sb\infty$ and for each $i\in\N$ there exists $j\geq i$ such that  
    $g_{i}^{j}\left(B\left(x_j,2r_j\right)\right)\subset B\left(x_i,r_i\right)$.
    \item The inverse sequence  $\left\{\left(G_n,r_n\right),g_{n}^{n+1}\right\}$ is a g-cell structure.
\end{enumerate}
\end{proposition}
\begin{proof}
To prove (1) implies (2), let $\bar{x}\in G\sb\infty$ and let $i\in\N$. Take $j=i$. Since $r_j$ is transitive, by hypothesis, then $B(x_j,2r_j)=B(x_j,r_j)$. Since $g_i^i$ is the identity, this proves that  $(2)$ holds.

Now, we prove (2) implies (3). By Proposition~\ref{relation}, the natural relation is reflexive and symmetric. To check that it is a g-cell structure, the transitivity of the natural relation is the only property missing. The proof  follows from the same argument given in the proof of  \cite[Lemma~3.1]{DT}, which we repeat here for completeness sake. Suppose that  $(\bar{x},\bar{y}),(\bar{y},\bar{z})\in r$ and let $i\in\N$. Then there is   $j\geq i$ such that $g_{i}^{j}\left(B\left(x_j,2r_j\right)\right)\subset B\left(x_i,r_i\right)$. Since $(x_j,y_j),(y_j,z_j)\in r_j$ then $z_j\in B(x_j,2r_j)$. Therefore $z_i=g_i^j(z_j)\in B(x_i,r_i)$, that is, $(x_i,z_i)\in r_i$. Thus $(\bar{x},\bar{z})\in r$, as wanted. 
\end{proof}

Condition (2) in previous Proposition~\ref{Prop:implyGcell}  may be interpreted as follows: For each $\bar{x}\in G\sb\infty$ and each $i\in\N$, there exits  $j\geq i$ such that $g_i^j$ collapses any path of length two  beginning  at the vertex $x_j$ and through three different vertices  to an edge beginning at $x_i$.

Given an inverse sequence of cellular graphs, to prove that it is a g-cell structure will usually be verified by checking condition (2) in Proposition~\ref{Prop:implyGcell}.

Notice that Definition~\ref{GCellStruc} allows the spaces $G_n$ to be not necessary discrete. It would be trivial to propose a g-cell structure of any topological space $X$ by considering the inverse sequence $G_n=X$, $r_n=\Delta$ and the bonding maps $g_{n}^{n+1}$ are all the identity maps. Therefore, to avoid such trivialities, it becomes important to require either that every $G_n$ has the discrete topology or that not all $r_n$ are equal to $\Delta$. The role of the bonding maps, evidently, becomes predominant, as the following examples show.

\begin{example}\label{ex:gcell}
For each $i\in\N$, consider $G_i=[0,1]$ with the standard topology and define
\[
r_i=\Delta \cup\left\{ (x,1-x),(1-x,x)\colon 0<x<\tfrac{1}{2}\right\}\cup\left\{(0,\tfrac{1}{2}),(\tfrac{1}{2},0),(\tfrac{1}{2},1),(1,\tfrac{1}{2})\right\},
\]
while the bonding maps are given by
\[
g_{n}^{n+1}(x)=
    \begin{cases}
        x &\text{ if } 0\leq x < \frac{1}{2}\\
        1-x &\text{ if } \frac{1}{2}\leq x \leq 1
    \end{cases}.
\]
Then $\left\{\left(G_n,r_n\right),g_{n}^{n+1}\right\}$ is a g-cell structure while none of the $G_n$'s is a discrete spaces. 
\end{example}

Observe that Example~\ref{ex:gcell} also exhibits that the converse of the implication $(1)\Rightarrow (2)$ in Proposition~\ref{Prop:implyGcell} is false. The following example will show how a modification on the bonding maps may produce the loss of the cell structure on a sequence of cellular graphs.

\begin{example}\label{ex:noGcell}
   Consider the cellular graphs $\{(G_n,r_n)\}$ given in Example~\ref{ex:gcell}. Now  redefine all the bonding maps $g_{n}^{n+1}$  as the identity maps. Then the resulting inverse sequence   is no longer a cell structure. Indeed, the constant sequences $\bar{0}=(0,0,0, \ldots)$, 
   $\bar{\frac{1}{2}}=(\frac{1}{2},\frac{1}{2},\frac{1}{2}, \ldots)$ and
   $\bar{1}=(1,1,1, \ldots)$ belong to $G\sb\infty$. Furthermore, $(\bar{0},\bar{\frac{1}{2}})\in r$
   and $(\bar{\frac{1}{2}},\bar{1})\in r$. However, $(\bar{0},\bar{1})\not\in r$. Thus $r$ is not an equivalence relation. 
\end{example}


\begin{definition}
Let $\left\{\left(G_n,r_n\right),g_{n}^{n+1}\right\}$ be a g-cell structure. Let $G\sb\infty=\displaystyle\lim_{\longleftarrow} \{G_n,g_{n}^{n+1}\}$ and let $r$ be the induced relation on $G\sb\infty$ as given in Definition~\ref{natrel}. 
We denote the quotient space $G^\ast$ as
\begin{equation*}
    G\sp\ast=G\sb\infty\slash r.
\end{equation*}
We will say that a topological space $X$ admits a g-cell structure if there exists a g-cell structure such that $X$ is homeomorphic to $G\sp\ast$.
\end{definition}

In Example~\ref{ex:gcell} one may show that $G\sb\infty $ is homeomorphic to $[0,\tfrac{1}{2}]$ with the usual topology and that $(\bar{0},\bar{1})$ is the only nontrivial relation in $r$. Thus $G\sb\infty$ is homeomorphic to $S^1$. This gives a g-cell structure for $S^1$ alternative to cell structure obtained in \cite{DT}.

 When every $r_j$ is an equivalence relation, one might think that the inverse limit of $G_j^\ast=G_j\slash r_j$ is equal to $G\sp\ast$. This is not true, as the following example shows.

\begin{example}\label{ex:differentquotient}
   Consider the discrete spaces $G_n=(0,\tfrac{1}{2}]\cap\Q$ together with $r_n=G_n\times G_n$.  Define now the bonding maps $g_{n}^{n+1}$ to be the inclusion maps. Then $G\sp\ast$ is empty, while the inverse limit of $G_{n}^{\ast}$ is the one point space.
\end{example}
The above Example~\ref{ex:differentquotient} may be slightly modified by taking $G_n=\N$, $r_n=G_n\times G_n$ and all bonding maps to be the identity. Then again $G\sp\ast=\N$ is different from the inverse limit of the one-point space $G_{n}^{\ast}$.

\begin{proposition}\label{open}
Let $\bigl\{(G_n, r_n),g_{n}^{n+1}\bigr\}_{n\in\N}$ be a  g-cell structure. The quotient map  $\pi \colon G\sb\infty\to G\sp\ast$ is closed if and only if for every $\bar{x}\in G\sb\infty$
and every open neighbourhood $A\subset G\sb\infty$ of $B(\bar{x},r)$
there exists  $U$ in the base of $G\sb\infty$ such that $\bar{x}\in U$ and
$B\left(U,r\right)\subset A$.
\label{QM_Closed}
\end{proposition}

\begin{proof}
Let $\bar{x}\in G\sb \infty$ and let $A\subset G\sb \infty$ be an open subset in $G\sb\infty$ such that $B(\bar{x},r)\subset A$. Define 
\begin{align*}
   O&=\{\bar{z}\in G\sb\infty
        \colon \pi(\bar{z})\cap A\not=\emptyset
        \}
    =\{\bar{z}\in G\sb \infty: B(\bar{z},r)\subset A\}. 
\end{align*}
If $\pi$ is a closed map, by \cite[Proposition 2.4.9]{E}, the set $O$ is an open set and since $\bar{x}\in O$, by Proposition~\ref{basis}, there exists $U$ in the base of $G\sb\infty$ such that $\bar{x}\in U\subset O$. By definition of $O$, for each $\bar{y}\in U$, $B(\bar{y},r)\subset A$. Thus, $B(U,r)\subset A$.  For the converse, if $\bar{z}\in O$ then, by  hypothesis, there exists a basis element $U$ such that $\bar{z}\in U$ and  $B\left(U,r\right)\subset A$. Let $\bar{y}\in U$ then $B(\bar{y},r)\subset B(U,r)\subset A$ and thus $\bar{y}\in O$. Consequently $U\subset O$ and so $O$ is an open set. By \cite[Proposition 2.4.9]{E}, $\pi$ is a closed map.
\end{proof}

In Propositon~\ref{open}, if  $G_n$ has the discrete topology for all $n$, we obtained a condition similar to  \cite[Proposition~3.8]{DT} for cell structures as the following result shows.

\begin{corollary}
 Let $\bigl\{(G_n, r_n),g_{n}^{n+1}\bigr\}_{n\in\N}$ be a  g-cell structure where every $G_n$ has the discrete topology. If for every $\bar{x}\in G\sb\infty$
and every open neighbourhood $A\subset G\sb\infty$ of $B(x,r)$
there exists $j\in\N$ such that $B(g_{j}^{-1}(x_j), r)\subset A$ then the collection
 \begin{equation*}
    \{G\sp\ast\setminus \pi\left(G\sb\infty\setminus \mathcal V 
    \right) \colon \mathcal V \text{ is open in } 
    G\sb\infty\}
\end{equation*}
is a basis of open sets for the topology of $G\sp\ast$.
\end{corollary}

 \begin{proof}
Let $[\bar{x}]\in G^{*}$ and let $U\subset G^{*}$ be an open set containing $[\bar{x}]$. Since $\pi^{-1}(U)$ is an open set in $G_{\infty}$, then by Proposition~\ref{basis}, for each $\bar{y}\in B(\bar{x},r)$ there exists $i_{\bar{y}}$ such that $g_{i_{\bar{y}}}^{-1}(y_{i_{\bar{y}}})\subset \pi^{-1}(U)$. Let $\mathcal{V}=\bigcup \left\{g_{i_{\bar{y}}}^{-1}(y_{i_{\bar{y}}}): \bar{y}\in B(\bar{x},r) \text{ and } g_{i_{\bar{y}}}^{-1}(y_{i_{\bar{y}}})\subset \pi^{-1}(U)\right\}$, then $\mathcal{V}$ is an open set $G_{\infty}$ such that $B(\bar{x},r)\subset \mathcal{V}\subset \pi^{-1}(U)$. By  Proposition \ref{QM_Closed}, $\pi(G_{\infty}\backslash \mathcal{V})$ is a closed set and thus $\mathrm{G^{*}} \backslash \pi(G_{\infty}\backslash \mathcal{V})$ is an open set. On the other hand we have $G_{\infty}\backslash \mathcal{V} \subset G_{\infty}\backslash B(\bar{x},r)$. This mean that $B(\bar{x},r)\cap (G_{\infty}\backslash \mathcal{V})= \emptyset$ and so $[\bar{x}]\notin \pi (G_{\infty}\backslash \mathcal{V})$. Thus $[\bar{x}]\in \mathrm{G^{*}} \backslash \pi (G_{\infty}\backslash \mathcal{V})$. Finally, since $G_{\infty}\backslash \pi^{-1}(U)\subset G_{\infty}\backslash \mathcal{V}$, then $\pi(G_{\infty}\backslash \pi^{-1}(U))\subset \pi(G_{\infty}\backslash \mathcal{V})$, so that $G\sp\ast\setminus\pi(G_{\infty}\backslash \mathcal{V})\subset G\sp\ast\setminus\pi(G_{\infty}\backslash \pi^{-1}(U))$. Now, since $\pi (\pi^{-1}(G^{*} \backslash U))=G^{*} \backslash U$ and $G\sb\infty\setminus\pi^{-1}
(U)=\pi^{-1}(G\sp\ast\setminus U)$ then 
$\pi(G_{\infty}\backslash \pi^{-1}(U))= G^{*} \backslash U$. We conclude that  $G^{*} \backslash \pi(G_{\infty}\backslash \mathcal{V})\subset U$. 
\end{proof}


\begin{lemma}
Let $\bigl\{(G_{i},r_{i}),g_{i}^{i+1}\bigr\}_{i\in \N}$ be a g-cell structure where each $G_{i}$ satisfies that for each $x\in G_{i}$ and for each open set $U$ containing $B(x,r_{i})$, there exists an open set  $O$ such that $x\in O$ and $B(O,r)\subset U$. Let $i\in \N$ and let $\bar{x}=(x_{1},x_{2},x_{3},...)\in G_{\infty}$. Then $A_{x_{i},U}=\bigl\{[\bar{z}]\in G^{*}:z_{i}\in U \text{ and there exists } j>i \text{ such that } (z_{j}, w_{j})\notin r_{j}, \forall \bar{w}\in g_{i}^{-1}\left(B(U,r_{i}\right)\backslash U)\bigr\}$ is an open set in $G^{*}$ and $\pi(\bar{x})\in A_{x_{i},U}$.
\label{LOpenSet}
\end{lemma}
        
\begin{proof}
We will show that the condition in $A_{x_{i},U}$ is independent of the representative of the class. Let $\bar{z}\in G_{\infty}$ such that $z_{i}\in U$ and $j>i$ satisfying $(z_{j},y_{j})\notin r_{j}$ for all $\bar{y}\in g_{i}^{-1}(B(U,r_{i})\backslash U)$. Let $\bar{w}\in G_{\infty}$ such that $(\bar{w},\bar{z})\in r$. Since $z_{i}\in U$, there exists an open set $O$ such that $z_{i}\in O$ and $B(O,r)\subset U$. Hence, since $w_{i}\in B(O,r)$, then $w_{i}\in U$. By the definition of g-cell structure, there exits $k>j$ such that $g_{j}^{k}(B(z_{k},2r_{k}))\subset B(z_{j},r_{j})$ and suppose that there exists $\bar{y}\in g_{i}^{-1}(B(U,r_{i}))\backslash U$ such that $(y_{k},w_{k})\in r_{k}$. Then $y_{k}\in B(z_{k},2r_{k})$, which implies $(z_{j},y_{j})\in r_{j}$. This is a contradiction on the choosing of $j$. Thus, for all $\bar{y}\in g_{i}^{-1}(B(U,r_{i})\backslash U)$, $(y_{k},w_{k})\notin r_{k}$, this mean that $\bar{w}$ satisfies the condition in $A_{x_{i},U}$.
        	
Now, we will prove that $A_{x_{i},U}$ is an open set. Let $\bar{z}\in \pi^{-1}(A_{x_{i},U})$, then $z_{i}\in U$ and there exists $j>i$ such that $(z_{j},y_{j})\notin r_{j}$ for all $\bar{y}\in g_{i}^{-1}(B(U,r_{i})\backslash U)$. Since $z_{i}\in U$, then there exists an open set $O$ such that $z_{i}\in O$ and $B(O,r_{i})\subset U$. Then $\bar{z}\in g_{i}^{-1}(O)$. We will show that $g_{i}^{-1}(O)\subset A_{x_{i},U}$. Let $\bar{u}\in g_{i}^{-1}(O)$, then $u_{i}\in O\subset U$. Suppose that there exists $\bar{y}\in g_{i}^{-1}(B(U,r_{i})\backslash U)$ such that $(y_{j},u_{j})\in r_{j}$. Then $(y_{i},u_{i})\in r_{i}$. Hence, $y_{i}\in B(u_{i},r_{i})\subset B(O,r_{i})\subset U$. This is a contradiction because $y_{i}\in B(U,r_{i})\backslash U$. Consequently, $(y_{j},u_{j})\notin r_{j}$ for all $\bar{y}\in g_{i}^{-1}(B(U,r_{i})\backslash U)$. Then $[\bar{u}]\in A_{x_{i},U}$ which implies that $\bar{u}\in \pi^{-1}(A_{x_{i},U})$. Hence, $\pi^{-1}(A_{x_{i},U})$ is an open set, so that $A_{x_{i},U}$ is an open set in $G^{*}$.
\end{proof}

\section{Results}

In this section we present the main results of this work. Since we do not require that for each thread $\bar{x}$ and each $i\in N$ there is $j\geq i$ such that $g_i^j(B(x_i,r_i)$ is finite (a condition that implies that the quotient map $G\sb\infty \mapsto G\sp\ast$ to be a perfect map), we begin by giving alternative condition on the g-cell structure which implies $G\sp\ast$ to be Hausdorff or normal.

\begin{theorem}\label{Hausdorff}
Let $\bigl\{(G_{i},r_{i}),g_{i}^{i+1}\bigr\}_{i\in \N}$ be a g-cell structure where each $G_{i}$ is a Hausdorff space and satisfies that for each $x\in G_{i}$ and for each open set $U$ containing $B(x,r_{i})$, there exists an open set  $O$ such that $x\in O$ and $B(O,r_i)\subset U$. Then the quotient space  $G^{*}$ is a Hausdorff space. 
\end{theorem}

\begin{proof}
Let $\pi(\bar{x})$, $\pi(\bar{y})$ in $G^{*}$ such that $\pi(\bar{x})\neq \pi(\bar{y})$ and $\bar{x},\bar{y} \in G_{\infty}$. Then $(\bar{x},\bar{y})\notin r$ and thus there exists $i\in \N$ such that $(x_{i},y_{i})\notin r_{i}$. Since that $G_{i}$ is a Hausdorff space there exists disjoint open sets $U$ and $V$ such that $x_{i}\in U$ and $y_{i}\in V$. Define the following sets:
\begin{equation*}
    A_{x_{i},U}=\bigl\{[\bar{z}]\in G^{*}:z_{i}\in U \text{ and } \exists j>i \text{ such that } (z_{j}, w_{j})\notin r_{j}, \forall \bar{w}\in g_{i}^{-1}\left(B(U,r_{i})\backslash U\right)\bigr\},
\end{equation*}
\begin{equation*}
    B_{y_{i},V}=\bigl\{[\bar{z}]\in G^{*}:z_{i}\in V \text{ and } \exists j>i \text{ such that } (z_{j}, w_{j})\notin r_{j}, \forall \bar{w}\in g_{i}^{-1}\left(B(V,r_{i})\backslash V\right)\bigr\}.
\end{equation*}
By the Lemma \ref{LOpenSet}, $A_{x_{i},U}$ and $B_{y_{i},V}$ are open sets in $G^{*}$ containing $\pi(\bar{x})$ and $\pi(\bar{y})$, respectively. We will show that $A_{x_{i},U}$ and $B_{y_{i},V}$ are disjoint sets. Suppose that there exists $[\bar{z}]\in G^{*}$ such that $[\bar{z}]\in A_{x_{i}}\cap B_{y_{i}}$, then there exists $\bar{w}$ and $\bar{u}$ in $[\bar{z}]$ such that $\bar{w}$ satisfies the condition in $A_{x_{i},U}$ and $\bar{u}$ satisfies the condition in $B_{y_{i},V}$. Since $(w_{i},u_{i})\in r_{i}$, then $w_{i}\in B(V,r_{i})$. But, $U\cap V=\emptyset$ then $w_{i}\in B(V,r_{i})\backslash V$ and so $\bar{w}\in g_i^{-1}(B(V,r_i)\setminus V$. Hence, $(w_{j},u_{j})\notin r_{j}$ for some $j$. This is a contradiction because $\bar{u}$ and $\bar{w}$ are in the same equivalence class.
\end{proof}

\begin{corollary}
 If $\left\{\left(G_i,r_i\right),g_{i}^{i+1}\right\}$ is a g-cell structure where each $G_{i}$ is a discrete space then $G\sp\ast$ is a Hausdorff space. 
\end{corollary}

\begin{proof}
Since $G_{i}$ is a discrete space, then it is a Hausdorff space. Also, if $x\in G_{i}$ and $U$ is an open set containing $B(x,r_{i})$, then $O=\{x\}$ is an open set containing $x$ and $B(O,r_{i})\subset U$. By the Theorem \ref{Hausdorff}, $G^{*}$ is a Hausdorff space.
\end{proof}

\begin{theorem}\label{normal}
Let  $X$ be a space admitting a g-cell structure. Then $X$ is normal provided its g-cell structure satisfies the  following properties: (i) Every $G_n$ is a metric space. (ii)  For every 
$\bar{x}\in G\sb\infty$ and for every open neighborhood $A\subset G\sb\infty$ of $B(\bar{x}, r)$ there exists 
 $U$ in the base of $G\sb\infty$ such that $\bar{x}\in U$ and  $B(U,r)\subset A$.
\end{theorem}

\begin{proof}
By Proposition~\ref{QM_Closed}, the quotient map $\pi:G\sb\infty\to G^{*}$ is closed. Since every $G_n$ is metric then $\prod_{n\in\N} G_n$ is also metric, \cite[Theorem 4.2.2, pg. 259]{E}. Therefore
$G\sb \infty$ is a normal space and
applying \cite[Theorem 1.5.20, pg.~46]{E} we get that $G^{*}$ is a  normal space and thus $X$ is a normal space.
\end{proof}

\begin{theorem}\label{wedge}
 Let $\{X_k\}_{k\in\N}$ be a sequence of topological spaces each admitting a cellular structure and without isolated points. Then the wedge sum $X=\bigvee_{k\in\N} X_k$ does not admit a cell structure but does admit a g-cell structure.
\end{theorem}

\begin{proof}
    
For each $k\in \N$, let us denote $p_{k}\in X_{k}$ the distinguished base point in $X_{k}$ and let $q:\bigsqcup_{k\in \N}X_{k}\to \bigvee_{k\in\N} X_k$ the quotient map which identifies all the points $p_{k}$ into a single point $x_0$.
Since each $X_{k}$ is a metrizable space (in fact completely me\-tri\-za\-ble by \cite[Theorem~3.6]{DT}), let $\rho_{k}$ be a metric in $X_{k}$, for each $k\in \N$. Suppose that there exists a countable neighborhood base $\{u_{i}\}_{i\in \N}$ of the point $x_{0}$. Since $q^{-1}(u_{k})$ is an open set in $\sqcup_{k\in \N}X_{k}$, then $q^{-1}(u_{k}) \cap X_{k}$ is an open set in $X_{k}$. Also, $q^{-1}(u_{k}) \cap X_{k}$ is a non-degenerate set because $x_{0}\in q^{-1}(u_{k}) \cap X_{k}$ and $X_{k}$ is a topological space without isolated points. Then there exists $y_{k}\in q^{-1}(u_{k}) \cap X_{k}\backslash \{p_{k}\}$. Put 
$d_{k}=\rho_{k}(p_{k},y_{k})$ and denote $B_k=\{y_k\in X_k\colon \rho_k(y_k,p_k)<d_k\}$ the open ball with center $p_k$ and radius $d_k$.  By definition of the mapping $q$ we have $q^{-1}\left(q\left(\sqcup B_k\right)\right)=\sqcup B_{k}$, then by definition of quotient topology, we have that $q\left(\sqcup B_k\right)$ is an open set in $X=\bigvee_{k\in \N} X_{k}$ which contains $x_0$. We claim that $q(\sqcup B_k)$  does not contain any of the elements of the base $\{u_{i}\}_{i\in \N}$. Indeed, if $u_{m}\subset q(\sqcup B_{k})$, then  $q^{-1}(u_{m})\subset q^{-1}(q(\sqcup B_{k}))=\sqcup B_{k}$. Hence, $y_m\in q^{-1}(u_{m})\cap X_{m}\subset B_m$, a contradiction. Thus, $x_{0}$ does no admit a countable neighborhood base. Then $X$ is not a first countable space. We conclude that $X$ is not a metrizable space. Thus X  does not admit a cell structure by \cite[Theorem~3.6]{DT}.

For each $k\in \N$, let $(\ast)_{k}=\{(G_{i}^{k},r_{i}^{k}),(g_{k})_{i}^{i+1}\}$ be a cell structure for $X_k$ built as in \cite[Theorem 4.3]{DT}. The cell structure $(\ast)_{k}$ determines a inverse limit which will be denoted by $G_{\infty}^{k}$ and $G_{k}^{*}$ will denote the space  determined by the cell structure. The $i^{th}$ projection map restricted to $G_{\infty}^{k}$ will be denoted by $(g_{k})_{i}$.  The quotient map of $G_{\infty}^{k}$ onto $G_{k}^{*}$ will be denoted by $\pi_{k}$ while $\varphi_{k}$ will denote the homeomorphism of $G_{k}^{*}$ onto $X_{k}$.

For each $k\in\N$, fix $\bar{x}_k=(x_{1}^{k},x_{2}^{k},\ldots) \in (\phi_k\circ\pi_k)^{-1}(p_k)$. Furthermore, since $(*)_k$ is a cell structure for $X_k$, we may choose inductively $j_{i}^{k}>j_{i-1}^{k}$ the least natural number such $(g_k)_{j_{i-1}^k}^{j_i^k}\bigl(B\bigl(x_{j_i^k}^k,2r_{j_i^k}^k\bigr)\bigr)\subset B\bigl(x_{j_{i-1}^k}^k,r_{j_{i-1}^k}^k\bigr)$, where $j_0^k=1$. We define the following sets:
\begin{itemize}
           \item $G_{i}=\bigcup_{k\in \mathbb{N}} G_{j_i^k}^{k}$
           \item $A_{i}=\bigcup_{k\in \mathbb{N}} r_{j_{i}^{k}}^{k}$
           \item $B_{i}=\bigcup_{k,t} B\left((g_k)_{j_i^k}(\varphi_k\circ \pi_k)^{-1}(p_k),r_{j_i^k}^k\right)\times B\left((g_t)_{j_i^t}(\varphi_t\circ \pi_t)^{-1}(p_t),r_{j_i^t}^t\right)$
\end{itemize}
For each $i\in \mathbb{N}$, we will consider on $G_{i}$ the relation $R_i=A_i\cup B_i$. Notice that each $R_i$ is a reflexive relation, because $A_i$ contains to the diagonal of $G_i\times G_i$. Also, $R_i$ is a symmetric relation, since $A_i$ and $B_i$ are symmetric. Thus, for each $i\in \mathbb{N}$, $(G_i,R_i)$ is a cellular graph.
We define, for each $i\in \N$, the bonding maps as follow: let $f_{i}^{i}$ be the identity on $G_i$ and $f_{i}^{i+1}:G_{i+1} \to G_{i}$ defined by
       \begin{equation*}
           f_{i}^{i+1}(x)= 
    			(g_{k})_{j_{i}^k}^{j_{i+1}^k}(x), \text{  } x\in G_{j_{i+1}^k}^{k}.
       \end{equation*}
       We will prove that $\bigl(f_{i}^{i+1}(x),f_{i}^{i+1}(y)\bigr)\in R_{i}$ whenever $(x,y)\in R_{i+1}$.
       \begin{itemize}
           \item If $(x,y)\in A_{i+1}$, then $(x,y)\in r_{j_{i+1}^k}^k$ for some $k\in \mathbb{N}$. Since $(g_k)_{j_{i}^{k}}^{j_{i+1}^k}$ is a composition of bonding maps, we get that $\left(f_{i}^{i+1}(x),f_{i}^{i+1}(y)\right) \in r_{j_{i}^{k}}^k$. Therefore, $\bigl(f_{i}^{i+1}(x),f_{i}^{i+1}(y)\bigr)\in A_i\subset R_i$.
           \item If $(x,y)\in B_{i+1}$, there exist $l,k\in \mathbb{N}$ such that $x\in B\bigl((g_k)_{j_{i+1}^k}(\varphi_k\circ \pi_k)^{-1}(p_k),r_{j_{i+1}^k}^k\bigr)$ and $y\in B\bigl((g_l)_{j_{i+1}^l}(\varphi_l\circ \pi_l)^{-1}(p_l),r_{j_{i+1}^l}^l\bigr)$. Thus, there is $z=(g_k)_{j_{i+1}^k}(\bar{z})$ such that $\bar{z}=(\varphi_k\circ\pi_{k})^{-1}(p_k)$ and $(x,z)\in r_{j_{i+1}^k}^k$. This implies that $f_{i}^{i+1}(x)=(g_{k})_{j_{i}^{k}}^{j_{t+1}^{k}}(x)$ and $f_{i}^{i+1}(z)=(g_{k})_{j_{i}^{k}}^{j_{t+1}^{k}}(z)$, therefore $\bigl(f_{i}^{i+1}(x),f_{i}^{i+1}(z)\bigr)\in r_{j_{i}^{k}}^{k}$. Also, since $f_{i}^{i+1}(z)=(g_{k})_{j_{i}^{k}}^{j_{i+1}^{k}}\bigl((g_k)_{j_{i+1}^{k}}(\bar{z})\bigr)=(g_{k})_{j_{i}^{k}}(\bar{z})$, we have $f_i^{i+1}(x)\in B\bigl((g_k)_{j_{i}^{k}}(\varphi_k\circ \pi_k)^{-1}(p_k),r_{j_i^k}^k\bigr)$. In the same way, we get that $f_i^{i+1}(y)\in B\bigl((g_l)_{j_{i}^{l}}(\varphi_l\circ \pi_l)^{-1}(p_l),r_{j_i^k}^l\bigr)$. Hence, we conclude that $\bigl(f_i^{i+1}(x),f_i^{i+1}(y)\bigr)\in B_{i}\subset R_i$.
       \end{itemize}
       Therefore $\bigl\{\bigl(G_i,R_i\bigr),f_i^{i+1}\bigr\}$ is a inverse sequence of cellular graphs and we will denote $G\sb\infty$ its inverse limit. Now to show that the condition of g-cell structure is satisfied, let $\bar{x}\in G_\infty$ and let $i\in \mathbb{N}$. Then, for some $t\in \mathbb{N}$, $\bar{x}\in \lim\limits_{\longleftarrow} \{F_{i}^{t},f_{i}^{i+1}\}:=F_{\infty}^t$, where we set, for each $i\in \mathbb{N}$, $F_i^t = G_{j_i^t}^t$. Notice that $G_{\infty}^{t}=F_{\infty}^{t}$, because the set $\{j_s^t: s\in \mathbb{N}\}$ is cofinal in $\mathbb{N}$. Then, there is $\bar{x}'=(x_1',x_2',\ldots,x_i',\ldots)\in G_{\infty}^{t}$ which corresponds to $\bar{x}=\left(x_{j_1^t}^t,x_{j_2^t}^t,\ldots,x_{j_i^t}^t,\ldots\right)\in \lim\limits_{\longleftarrow} \{F_{i}^{t},f_{i}^{i+1}\}$. For $i'=j_i^t\in \mathbb{N}$, let $k> j_{i+2}^t>i$ such that $(g_t)_{i'}^{j_k^t}(B(x'_{j_k^t},2r_{j_k^t}^t))\subset B(x'_{i'}\,{},r_{i'}^t)$. We will show that $f_i^{k}(B(x_{k},2R_{k}))\subset B(x_i,R_i)$. Let $y\in B(x_{k},2R_{k})$. 
      Notice that $2R_{k}=2A_k \cup 2B_{k} \cup (A_k+B_k) \cup (B_k+A_k)$. 
       	      \begin{itemize}
       	          \item If $(x_k,y)\in 2A_k$, then there exists $z\in G_k$ such that $(x_k,z)\in A_k$ and $(z,y)\in A_k$. Since $x_k\in G_{j_k^t}^t$, then $(x_k,z)\in r_{j_k^t}^t$ and necessarily $z\in G_{j_k^t}^t$. Thus $(z,y)\in r_{j_k^t}^t$. This implies that $y\in B\bigl(x_{k},2r_{j_k^t}^t\bigr)=B\bigl(x'_{j_k^t},2r_{j_k^t}^t\bigr)$. Therefore $f_i^k(y)=(g_t)_{i'}^{j_k^t}(y)\in B(x'_{i'},r_{i'}^t)=B(x'_{j_i^t},r_{j_i^t}^t)=B(x_i,r_{j_i^t}^t)\subset B(x_i,A_i)$. We get that $(f_i^k(y),f_i^k(x_k))\in A_i\subset R_i$.
       	          \item If $(x_k,y)\in 2B_k$, since $B_k$ is transitive, we have $(x_k,y)\in B_k$ and so we obtain the result.
       	          \item If $(x_k,y)\in A_{k}+B_{k}$, there exists $z\in G_{k}$ such that $(x_k,z)\in A_{k}$ and $(z,y)\in B_{k}$. Since $(x_k,z)\in A_{k}$, then $(x_k,z)\in r_{j_k^t}^{t}$. Hence $\left((g_t)_{j_{i+1}^t}^{{j_k^t}}(x_k),(g_t)_{j_{i+1}^t}^{{j_k^t}}(z)\right)\in r_{j_{i+1}^t}^t$. On the other hand, we have that $z\in G_{j_k^t}^t$ and $y\in G_{j_k^s}^s$ for some $s\in \mathbb{N}$. Since $(z,y)\in B_k$, then $(g_t)_{j_{i+2}^t}^{j_{k}^{t}}(z)\in B\left((g_t)_{j_{i+2}^t}(\varphi_t\circ \pi_t)^{-1}(p_t),r_{j_{i+2}^t}^t\right)$ and $(g_s)_{j_{i}^s}^{j_{k}^{s}}(y)\in$ \mbox{} $B\left((g_s)_{j_{i}^s}(\varphi_s\circ \pi_s)^{-1}(p_s),r_{j_{i}^s}^s\right)$. Now, $((g_t)_{j_{i+2}^t}^{j_{k}^{t}}(z),w)\in r_{j_{i+2}^t}^t$ for some $w\in (g_t)_{j_{i+2}^t}(\varphi_t\circ \pi_t)^{-1}(p_t)$ and therefore $(w,(g_t)_{j_{i+2}^t}(\bar{x}_t))\in r_{j_{i+2}^t}^t$. This mean that $(g_t)_{j_{i+2}^t}^{j_{k}^{t}}(z)\in B((g_s)_{j_{i+2}^t}(\bar{x}_t)),2r_{j_{i+2}^t}^t)$ and by the hypothesis, $(g_t)_{j_{i+1}^t}^{j_{k}^{t}}(z)\in B((g_t)_{j_{i+1}^t}(\bar{x}_t)),r_{j_{i+1}^t}^t)$. We obtain $(g_t)_{j_{i+1}^t}^{{j_k^t}}(x_k)\in B\left((g_t)_{j_{i+1}^t}(\bar{x}_t),2r_{j_{i+1}^t}^t\right)$. Also, $g_{j_i^t}^{j_{i+1}^t}\left(B\left((g_t)_{j_{i+1}^t}(\bar{x}_t),2r_{j_{i+1}^t}^t\right)\right)\subset B\left((g_t)_{j_i^t}(\bar{x}_t),r_{j_i^t}\right)$. Then $f_i^k(x_k)=g_{j_i^t}^{j_k^t}(x_k)\in B\left((g_t)_{j_i^t}(\bar{x}_t),r_{j_i^t}^t\right)$.
       	          We conclude that $f_i^k(x_k)\in B\left((g_t)_{j_i^t}(\bar{x}_t),r_{j_i^t}^t\right)\subset B\left((g_t)_{j_i^t}(\varphi_t\circ \pi_t)^{-1}(p_t),r_{j_i^t}^t\right)$. Since $f_i^k(y)=(g_s)_{j_{i}^s}^{j_{k}^{s}}(y)\in B\left|((g_s)_{j_i^s}(\bar{x}_s),r_{j_i^s}^s\right)\subset B\left((g_s)_{j_i^s}(\varphi_s\circ \pi_s)^{-1}(p_s),r_{j_i^s}^s\right)$, we get that $\left(f_i^k(x_k),f_i^k(y)\right)\in B_{i}\subset R_i$.
       	          \item If $(x_k,y)\in B_{k}+A_{k}$, there exists $z\in G_{k}$ such that $(x_k,z)\in B_{k}$ and $(z,y)\in A_{k}$. Since $(z,y)\in A_{k}$, there exists $s\in \mathbb{N}$ such that $(z,y)\in r_{j_k^s}^{s}$. Hence $((g_s)_{j_{i+1}^s}^{{j_k^s}}(y),(g_s)_{j_{i+1}^s}^{{j_k^s}}(z))\in r_{j_{i+1}^s}^s$. In the other hands, we have that $z\in G_{j_k^s}^s$ and $x_k\in G_{j_k^t}^t$. Since $(x_k,z)\in B_k$, then $(g_s)_{j_{i+2}^s}^{j_{k}^{s}}(z)\in B((g_s)_{j_{i+2}^s}(\varphi_s\circ \pi_s)^{-1}(p_s),r_{j_{i+2}^s}^s)$ and $(g_t)_{j_{i}^t}^{j_{k}^{t}}(x_k)\in B((g_t)_{j_{i}^t}(\varphi_t\circ \pi_t)^{-1}(p_t),r_{j_{i}^t}^t)$. Thus $((g_s)_{j_{i+2}^s}^{j_{k}^{s}}(z),w)\in r_{j_{i+2}^s}^s$ where $w\in (g_s)_{j_{i+2}^s}(\varphi_s\circ \pi_s)^{-1}(p_s)$. Hence, $(w,(g_s)_{j_{i+2}^s}(\bar{x}_s))\in r_{j_{i+2}^s}^s$. This implies that $(g_s)_{j_{i+2}^s}^{j_{k}^{s}}(z)\in B((g_s)_{j_{i+2}^s}(\bar{x}_s)),2r_{j_{i+2}^s}^s)$ and by the hypothesis, we get that $(g_s)_{j_{i+1}^s}^{j_{k}^{s}}(z)\in B((g_s)_{j_{i+1}^s}(\bar{x}_s)),r_{j_{i+1}^s}^s)$. Notice that $(g_s)_{j_{i+1}^s}^{{j_k^s}}(y)\in B((g_s)_{j_{i+1}^s}(\bar{x}_s),2r_{j_{i+1}^s}^s)$. Also, we have $g_{j_i^s}^{j_{i+1}^s}(B((g_s)_{j_{i+1}^s}(\bar{x}_s),2r_{j_{i+1}^s}^s))\subset B((g_s)_{j_i^s}(\bar{x}_s),r_{j_i^s}^s)$. Therefore $f_i^k(y)=g_{j_i^s}^{j_k^s}(y)\in B((g_s)_{j_i^s}(\bar{x}_s),r_{j_i^s}^s)$.
       	          We conclude that $f_i^k(y)\in B((g_s)_{j_i^s}(\bar{x}_s),r_{j_i^s}^s)\subset B((g_s)_{j_i^s}(\varphi_s\circ \pi_s)^{-1}(p_s),r_{j_i^s}^s)$. Since $f_i^k(x_k)=(g_t)_{j_{i}^t}^{j_{k}^{t}}(x_k)\in B((g_t)_{j_i^t}(\bar{x}_t),r_{j_i^t}^t)\subset B((g_t)_{j_i^t}(\varphi_t\circ \pi_t)^{-1}(p_t),r_{j_i^t}^t)$, we get that $(f_i^k(x_k),f_i^k(y))\in B_{i}\subset R_i$.
       	      \end{itemize}

\vspace{1cm}

We now define the following sets:
       	       \begin{equation*}
       	          A=\{(\bar{x},\bar{y})\in G\sb\infty\times G\sb\infty:(x_i,y_i)\in A_{i}, \text{ for each } i\in \mathbb{N}\}
       	      \end{equation*}
and
       	      \begin{equation*}
       	          B=\{(\bar{x},\bar{y})\in G\sb\infty\times G\sb\infty:(x_i,y_i)\in B_{i}, \text{ for each } i\in \mathbb{N}\}.
       	      \end{equation*}
       	      Notice that since $A$ is the induced relation on $G\sb\infty$ by the reflexive and symmetric relation $A_i$ on $G_i$, for each $i$, we have that $A$ is an equivalence relation. So we will consider the quotient space  $G_{\infty}/A$ (which is in fact homeomorphic to $\sqcup X_k$) and we will introduce and equivalence relation on $G\sb\infty\slash A$ as follows: 
\begin{eqnarray*}
B'=\bigl\{\bigl([\bar{x}]_A,[\bar{y}]_A\bigr)\in G\sb\infty\slash A\times G\sb\infty\slash A:  (\bar{x},\bar{y})\in B \bigr\}\cup \Delta.
\end{eqnarray*}
We will prove that the definition of $B'$ does not depend on the class representative. It will suffice to show that given $\bar{u}\in [\bar{x}]_A$ such that for each $i\in\N$ we have $x_i\in B\bigl((g_k)_{j_i^k}(\varphi_k\circ \pi_k)^{-1}(p_k),r_{j_i^k}^k\bigr)$ for some $k\in \mathbb{N}$, implies that for each $i\in \N$ we also have $u_i\in B\bigl((g_k)_{j_i^k}(\varphi_k\circ \pi_k)^{-1}(p_k),r_{j_i^k}^k\bigr)$ for some $k\in \mathbb{N}$. For this purpose, fix an arbitrary $i\in \N$. Since $x_i\in G_i$, there exists $k\in\N$ such that $x_i\in G_{j_i^k}^k$. Then  $(u_i,x_i)\in r_{j_i^{k}}^k$ and $(u_{i+1},x_{i+1})\in r_{j_{i+1}^k}$. Also, by assumption, there exists $\bar{w}\in G_\infty$ such that $\pi_k(\bar{w})=\varphi_k^{-1}(p_k)$ and $(x_{i+2},w_{j_{i+2}^k})\in r_{j_{i+2}^k}^k$.  Thus $x_{i+2}\in B\bigl(x_{j_{i+2}^{k}}^{k},2r_{j_{i+2}^{k}}^k\bigr)$ and this implies that $\bigl(x_{i+1},x_{j_{i+1}^{k}}^k\bigr)\in r_{j_{i+1}^{k}}^k$. Hence, $u_{i+1}\in B\bigl(x_{j_{i+1}^{k}}^k,2r_{j_{i+1}^{k}}^k\bigr)$ and so, $(u_{i},x_{j_{i}^{k}}^k)\in r_{j_{i}^{k}}^k$. We get that $u_i\in B\bigl((g_k)_{j_i^k}(\varphi_k\circ \pi_k)^{-1}(p_k),r_{j_i^k}^k\bigr)$, as wanted. Thus $B'$ is well defined. Notice that $B\sp\prime$ is an equivalence relation.
       	      Consider the funtion $\varphi:(G_\infty/A)/B'\to G_\infty/R$  defined as $\varphi \bigl(\bigl[[\bar{x}]_A\bigr]_{B'}\bigr)=[\bar{x}]_R$. We will show tha $\varphi$ is well defined.

\begin{figure}[ht]
\centering
\begin{tikzcd}
         G\sb\infty \arrow[drr,"\pi_{A\cup B}" '] \arrow[r,"\pi_A"] & G\sb\infty\slash A \arrow[r,"\pi_{B\sp\prime}"] & (G\sb\infty\slash A)\slash B\sp\prime \arrow[d,"\varphi"] \\
         & & G\sb\infty \slash R
\end{tikzcd}
\caption{The spaces $G\sb\infty \slash R$ and $(G\sb\infty\slash A)\slash B$ are homeomorphic}
\label{fig:cd}
\end{figure}

         Let $[\bar{y}]_A \in \bigl[[\bar{x}]_A\bigr]_{B'}$. If $\bigl([\bar{x}]_A,[\bar{y}]_A\bigr)\in B'\backslash \Delta$, then  $(\bar{x},\bar{y})\in B\subset R$. If $[\bar{x}]_A=[\bar{y}]_A$, then $(x_i,y_i)\in A_i$ and thus $(\bar{x},\bar{y})\in A\subset R$.   
      It follows that the  diagram in Figure~\ref{fig:cd} commutes. Hence $\varphi$ is continuous.
      We now show that $\varphi$  is injective. Let $\bigl([\bar{x}]_A,[\bar{y}]_A \bigr)\notin B'$, then $(\bar{x},\bar{y})\notin A$, since otherwise $[\bar{x}]_A=[\bar{y}]_A$ and so $\bigl([\bar{x}]_A,[\bar{y}]_A \bigr)\in\Delta\subset B'$, a contradiction. But also $(\bar{x},\bar{y})\notin B$ since otherwise
      $(\bar{x},\bar{y})\in B$ implies by definition that $\bigl([\bar{x}]_A,[\bar{y}]_A \bigr)\in B'$, a contradiction. 
      We conclude that $[\bar{x}]_R\neq [\bar{y}]_R$ and thus $\varphi$ is injective. It is clear that $\varphi$ is surjective and its inverse $[\bar{x}]_R\mapsto \bigl[[\bar{x}]_A\bigr]_{B'}$ is also continuous by an analogous argument using a commutative diagram.  Thus $\varphi$ is a homeomorphism

\end{proof}

In fact the space $X$ in Theorem~\ref{wedge} is a normal space. In the next result we prove that more is true: there is a non-regular space admitting a g-cell structure.

First, we build the g-cell structure which will determine a non-regular space.

\begin{definition}\label{Ej_No_regular}
Consider the following sets with the discrete topology.
\begin{equation*}
G_{1}=\{a_{1}\}\cup \left\{b_{1}^{k}:k\in \N\right\}, \text{  }G_{2}=\{a_{2}\}\cup \left\{b_{2}^{k}:k\in \N\right\}\cup C_{1}^{2}\cup \left\{d_{2}^{k}:k\in \N\right\}
\end{equation*}
where $C_{1}^{2}=\left\{c^{1}_{2,1},c^{2}_{2,1}\right\}$ and let $L_{2}=1$. In general, for each $i>2$ let $L_{i}=L_{i-1}+i-1$, $C_{k}^{i}=\left\{c^{1}_{i,k},c^{2}_{i,k}\right\}$ with $1\leq k \leq L_{i}$ and
\begin{equation*}
G_{i}=\bigl\{a_{i}\bigr\}\cup \bigl\{b_{i}^{k}: k\in \N\bigr\} \cup \left(\bigcup_{k=1}^{L_{i}} C_{k}^{i}\right) \cup \bigl\{d_{i}^{k}:k\in \N\bigr\}.
\end{equation*}
Now, for each $i\in \N$,
\begin{multline*}
r_{i}=\bigl\{\bigl(a_{i},b_{i}^{k}\bigr),\bigl(b_{i}^{k},a_{i}\bigr):k\geq i\bigr\} \cup \bigl\{\bigl(b_{i}^{k},b_{i}^{n}\bigr):k,n\geq i\bigr\} \cup 
\bigl\{(a,b):a,b\in C_{k}^{i}, k=1,...,L_{i}\bigr\}\\ \cup \bigl \{\bigl(d_{i}^{(i-1)j+k},b_{i}^{k}\bigr),\bigl(b_{i}^{k},d_{i}^{(i-1)j+k}\bigr): k<i, j\in \N\cup\{0\}\bigr\},
\end{multline*}
$g_{i}^{i}:G_{i}\to G_{i}$ is the identity function on $G_{i}$ and $g_{i}^{i+1}:G_{i+1}\to G_{i}$ is defined as

\[
g_{i}^{i+1}(a)=  
\begin{cases}
		a_{i} &    \text{if }  a=a_{i+1}  \hfill (1) \\ 
		 b_{i}^{k} &   \mbox{if } a=b_{i+1}^{k}\hfill (2) \\
		 a_{i} &    \text{if } a=c^{1}_{i+1,1}\hfill  (3)\\
		 b_{i}^{i} &  \text{if } a=c^{2}_{i+1,1} \hfill  (4)\\
		c^{r}_{i,k-1} &  \text{if } a=c^{r}_{i+1,k}, r=1,2, k=2,...,L_{i}+1 \hfill (5)\\
		d_{i}^{k-(L_{i}+1)} &  \text{if } a=c^{1}_{i+1,k}, k=L_{i}+2,...,L_{i+1} \hfill  (6)\\
		 b_{i}^{k-(L_{i}+1)} &  \text{if } a=c^{2}_{i+1,k}, k=L_{i}+2,...,L_{i+1}\hfill  (7)\\
		 d_{i}^{(i-1)(j+1)+n} &  \text{if } a=d_{i+1}^{k}, k=i\cdot j+n \text{ for } 0<n< i, j\in \N \cup\{0\}\phantom{p} \hfill  (8)\\
		 a_{i} &  \text{if } a=d_{i+1}^{k}, k=i\cdot j, j\in \N \cup\{0 \} \hfill (9)
\end{cases}
\]
\end{definition}

First we will prove the following lemma.

\begin{lemma}  \label{existe_c}
The sequence $\left\{\bigl(G_{i},r_{i}\bigr), g_{i}^{i+1}\right\}_{i\in \N}$ given in Definition~\ref{Ej_No_regular} is an inverse sequence of cellular graphs. Furthermore, given $\bar{w}=(w_{1},w_{2},w_{3},...)\in G_{\infty}$ such that $w_{i}=d_{i}^{k}$ for some $k,i\in \N$, then there exists $j_{0}\in \N$ such that $w_{j_{0}}\in C_{s}^{j_{0}}$, for some $s\in \N$.
\end{lemma}

\begin{proof}
By definition, $r_{i}$ is a reflexive and symmetric relation for each $i\in \mathbb{N}$. Also, by definition of the bonding functions, we get that  $\{(G_{i},r_{i}), g_{i}^{i+1}\}_{i\in \mathbb{N}}$ satisfies the first two conditions of inverse system.
We will show that the third condition is satistied. Let $a,b\in G_{i+1}$ such that $(a,b)\in r_{i+1}$.

\begin{itemize}
	\item If $a=a_{i+1}$ and $b=b_{i+1}^{k}$ with $k\geq i$, then $g_{i}^{i+1}(a)=a_{i}$ and $g_{i}^{i+1}(b)=b_{i}^{k}$. Hence, $(g_{i}^{i+1}(a),g_{i}^{i+1}(b))\in r_{i}$.
	\item If $a=b_{i+1}^{k}$ and $b=b_{i+1}^{n}$, then $g_{i}^{i+1}(b_{i+1}^{k})=b_{i}^{k}$ and $g_{i}^{i+1}(b_{i+1}^{n})=b_{i}^{n}$. Therefore, $(g_{i}^{i+1}(a),g_{i}^{i+1}(b))\in r_{i}$.
	\item If $a,b\in C_{k}^{i+1}$ we can suppose that  $a=c^{1}_{i+1,k}$ and $b=c^{2}_{i+1,k}$. If $k=1$, then $g_{i}^{i+1}(a)=a_{i}$ and $g_{i}^{i+1}(b)=b_{i}^{i}$. Thus, $(g_{i}^{i+1}(a),g_{i}^{i+1}(b))\in r_{i}$ .
	Now, if $k\in \{2,...,L_{i}+1\}$, we get that $g_{i}^{i+1}(a)=c^{1}_{i,k-1}$ and $g_{i}^{i+1}(b)=c^{2}_{i,k-1}$. Hence, $(g_{i}^{i+1}(a),g_{i}^{i+1}(b))\in r_{i}$. By last, if $k\in \{L_{i}+2,...,L_{i+1}\}$, then $g_{i}^{i+1}(a)=d_{i}^{k-(L_{i}+1)}$ and $g_{i}^{i+1}(b)=b_{i}^{k-(L_{i}+1)}$. Therefore $(g_{i}^{i+1}(a),g_{i}^{i+1}(b))\in r_{i}$.
	\item Suppose that $b=b_{i+1}^{k}, \text{ with } k<i+1, a=d_{i+1}^{i\cdot j+k} \text{ with } j\in \mathbb{N}\cup\{0\}$. If $i\cdot j+k=i\cdot r$ for some $r\in \mathbb{N}$, then $g_{i}^{i+1}(a)=a_{i}$ and $g_{i}^{i+1}(b)=b_{i}^{k}$. So, we get that $(g_{i}^{i+1}(a),g_{i}^{i+1}(b))\in r_{i}$. Now, if $i\cdot j+k=i\cdot r+n$ with $0<n<i$, then $g_{i}^{i+1}(a)=d_{i}^{(i-1)(j+1)+n}$ and $g_{i}^{i+1}(b)=b_{i}^{k}$. Since that $n<i$, we conclude that $(g_{i}^{i+1}(a),g_{i}^{i+1}(b))\in r_{i}$.
\end{itemize}

Consequently, we get that $\{(G_{i},r_{i}), g_{i}^{i+1}\}$ is an inverse system of cellular graphs.
Let $k\leq i-1$, notice that $w_{i+1}=c^{1}_{i+1,k+L_{i}+1}$, then we can take $j_{0}=i+1$ and $s=k+L_{j_{0}-1}+1$ and we get that $w_{j}\in C_{k+L_{j_{0}-1}+1}^{j}$. Now, we suppose that $k>i-1$. Then we can factorize $k$ as $k=(i-1)(j+1)+n$ for some $j\in \N\cup \{0\}$ and $0<n\leq i-1$. By definition of the function $g_{i}^{i+1}$ we get that $w_{i+1}=d_{i+1}^{i\cdot j+n}$. Now, we consider the preimage of $w_{i+1}$ under the function $g_{i+1}^{i+2}$. We have two cases again. If $i\cdot j+n\leq i$, we get that $w_{i+2}=c^{1}_{i+2,i\cdot j+n+L_{i+1}+1}$. Hence, we can take $j_{0}=i+2$ and $s=i\cdot j+n+L_{j_{0}-1}+1$. This implies that $w_{j_{0}}\in C_{i\cdot j+n+L_{j_{0}-1}+1}^{j_{0}}$. If $i\cdot j+n> i$, then $w_{i+2}=d_{i+2}^{(i+1)(j-1)+n}$. In general, if $w_{i+t}$ is not in some $C_{k}^{i}$, then $w_{i+t}=d_{i+t}^{(i+t-1)(j-t+1)+n}$. Taking $m=j+1$, we have $w_{i+m}=d_{i+m}^{n}$. Since $0<n<i-1\leq i+m-1$,  we get that $w_{i+m+1}=c^{i}_{i+m+1,n+L_{i+m}+1}$. Now, we can take $j_{0}=i+m+1$ and $s=n+L_{j_{0}-1}+1$ and we get that $w_{j_{0}}\in C_{n+L_{j_{0}-1}+1}^{j_{0}}$. Consequently, for each $\bar{w}\in G_{\infty}$ such that $w_{i}=d_{i}^{k}$ with $k,i\in \N$, there exists $j\in \N$ such that $w_{j}\in C_{s}^{j}$, for some $s\in \N$.
\end{proof}

We will also need the following.
\begin{lemma}\label{lemma:gcell}
The sequence $\left\{\bigl(G_{i},r_{i}\bigr), g_{i}^{i+1}\right\}_{i\in \N}$ given in Definition~\ref{Ej_No_regular} is a g-cell structure.
\end{lemma}

\begin{proof}
By Lemma~\ref{existe_c}, the system given in Definition~\ref{Ej_No_regular} is an inverse system.
It remains to prove that the  condition for a g-cell structure is satisfied.
Let $\bar{w}=(w_{1},w_{2},w_{3},...)\in G_{\infty}$ and $i\in \N$. 
If $w_{i}=a_{i}$ or $w_{i}=b_{i}^{k}$ with $k\geq i$, then $B(w_{i},r_{i})=\{a_{i}\}\cup \{b_{i}^{k}:k\geq i\}$ and $B(w_{i},2r_{i})=\{a_{i}\}\cup \{b_{i}^{k}:k\geq i\}$. Now, if $w_{i}=b_{i}^{k}$, with $k<i$, we get that $B(w_{i},r_{i})=\{b_{i}^{k}\}\cup \{d_{i}^{(i-1)j+k}:j\in \N\cup \{0\}\}$ and $B(w_{i},2r_{i})=\{b_{i}^{k}\}\cup \{d_{i}^{(i-1)j+k}:j\in \N\cup \{0\}\}$. If $w_{i}\in C_{k}^{i}$ for some $k\in \N$, then $B(w_{i},r_{i})=C_{k}^{i}$ and $B(w_{i},2r_{i})=C_{k}^{i}$. Hence, if $w_{i}\in \{a_{i}\}\cup \{b_{i}^{k}: k\in \N\}\cup \left(\bigcup_{k=1}^{L_{i}} C_{k}^{i}\right)$, then $B(w_{i},r_{i})=B(w_{i},2r_{i})$. Finally, if $\bar{w}\in G_{\infty}$ such that $w_{i}=d_{i}^{k}$ for some $k,i\in \N$. Then, by the Lemma \ref{existe_c}, there exists $j_{0}\in \N$ such that $w_{j_{0}}\in C_{s}^{j_{0}}$ for some $s\in \N$. Therefore, we have that $B(w_{j_{0}},r_{j_{0}})=B(w_{j_{0}},2r_{j_{0}})$. We conclude that for each $i\in \N$ and $\bar{w}\in G_{\infty}$, there exists $j\in \N$ such that $g_{i}^{j}(B(w_{j},2r_{j}))\subset B(w_{i},r_{i})$. Consequently, $\{(G_{i},r_{i}), g_{i}^{i+1}\}_{i\in \N}$ is a g-cell structure.
\end{proof}

We are now ready to prove our last result.
\begin{theorem}
 There are spaces admitting a g-cell structure which are not regular.
\label{No_regular}
\end{theorem}
\begin{proof}
Consider the inverse sequence of cell graphs
$\left\{\bigl(G_{i},r_{i}\bigr), g_{i}^{i+1}\right\}_{i\in \N}$ given in De\-fi\-ni\-tion~\ref{Ej_No_regular}. By  Lemma~\ref{lemma:gcell} it is a g-cell structure. Let  $\bar{a}=(a_{1},a_{2},a_{3},...)$, $A=\left\{(b_{1}^{k},b_{2}^{k},b_{3}^{k},...)\colon k\in \N\right\}$ and $B=\pi(A)$. We will show that $B$ is a closed set in $G^{*}$ and it does not contain $\pi(\bar{a})$. Let $b_{k}=(b_{1}^{k},b_{2}^{k},b_{3}^{k},...)\in A$ for some $k\in \N$. By definition of the relation $r_{k+1}$,  the element $a_{k+1}$ is related with $b_{k+1}^{t}$ if and only if $t\geq k+1$. Hence, $(b_{k+1}^{k},a_{k+1})\notin r_{k+1}$ this implies that $(b_{k},\bar{a})\notin r$. Since $k$ was chosen arbitrarily, $\bar{a}$ is not related with some element of $A$. Therefore, $\pi(\bar{a})\notin B$.
	
Let $\bar{x}=(x_{1},x_{2},x_{3},...) \in \pi^{-1}(G^{*}\backslash B)$. We shall show that $\pi^{-1}(G^{*}\backslash B)\subset G_{\infty}\backslash A$. Notice that if $\bar{b}\in \pi^{-1}(G^{*}\backslash B)$, then $\pi(\bar{b})\in G^{*}\backslash B$ and thus $\pi(\bar{b})\notin B=\pi(A)$. It result that $\bar{b}\notin A$ and this implies that $\bar{b}\in G_{\infty}\backslash A$. Therefore, $\bar{x}\notin A$. Now, we will show that there exists $k\in \N$ such that for every $n\in \N$, $x_{k}\neq b_{k}^{n}$. Suppose that for every $k\in \N$ there exists $n\in \N$ such that $x_{k}=b_{k}^{n}$. In particular for $k=1$, there exists $n_{1}$ such that $x_{1}=b_{1}^{n_{1}}$. Since that $x_{1}=g_{1}^{2}(x_{2})$ and for $k=2$ there exists $n_{2}$ such that $x_{2}=b_{2}^{n_{2}}$, then $x_{2}=b_{2}^{n_{1}}$. Applying induction, we suppose that $x_{t}=b_{t}^{n_{1}}$ for some $t\in \N$. Since $x_{t}=g_{t}^{t+1}(x_{t+1})$ and for $k=t+1$ there exists $n_{t+1}$ such that $x_{t+1}=b_{t+1}^{n_{t+1}}$, we get that $x_{t+1}=b_{t+1}^{n_{1}}$. Then $x_{t}=b_{t}^{n_{1}}$ for every $t\in \N$. Hence, $\bar{x}\in A$ which is a contradiction. So, there exists $k\in \N$ such that for each $n\in \N$, $x_{k}\neq b_{k}^{n}$.
	
We shall show that $g_{k}^{-1}(x_{k})$ is an open set and $\bar{x}\in g_{k}^{-1}(x_{k})\subset \pi^{-1}(G^{*}\backslash B)$. Let $\bar{w}=(w_{1},w_{2},w_{3},...)\in g_{k}^{-1}(x_{k})$, then $w_{k}=x_{k}$. Now, if $\bar{w}=\bar{a}$, then $\bar{w}\in \pi^{-1}(G^{*}\backslash B)$ because $\pi(\bar{a})\in G^{*}\backslash B$. Suppose that $\bar{w}\neq \bar{a}$, then $w_{k}=d_{k}^{n}$ or $w_{k}=c_{k,m}^{r}$ with $n,m\in \N$ and $r\in \{1,2\}$. First we will consider the case which $\bar{w}$ satisfies $w_{k}=d_{k}^{n}$ for some $n\in \N$. By the Lemma \ref{existe_c}, there exists $l>k$ such that $w_{l}=c_{l,t}^{r}$ for some $t\in \N$ and $r\in \{1,2\}$. Therefore, by definition of $r_{l}$, $(w_{l},b_{l}^{s})=(c_{l,t}^{r},b_{l}^{s})\notin r_{l}$ for every $s\in \N$. Then $(\bar{w},b_{s})\notin r$ for each $b_{s}\in A$ this implies that $\pi(\bar{w})\notin B$. Hence, $\bar{w}\in \pi^{-1}(G^{*}\backslash B)$. Now, we will consider $\bar{w}$ such that $w_{k}=c_{k,m}^{r}$ for some $m\in \N$ and $r\in \{1,2\}$. Then $(c_{k,m}^{r},b_{k}^{s})\notin r_{l}$ for every $s\in \N$. Similarly to the first case, we get that $\bar{w}\in \pi^{-1}(G^{*}\backslash B)$. We conclude that $g_{k}^{-1}(x_{k})$ is an open set such that $\bar{x}\in g_{k}^{-1}(x_{k})\subset \pi^{-1}(G^{*}\backslash B)$. 
Consequently $\pi^{-1}(G^{*}\backslash B)$ is open this implies that $G^{*}\backslash B$ is open in $G^{*}$. Therefore, $B$ is closed in $G^{*}$.
	
To show that $G^{*}$ is not regular, we will prove that there are not disjoint open sets $U$ and $V$ in $G^{*}$ such that $\pi(\bar{a})\in U$ and $B\subset V$. Let $U$, $V$ open sets in $G^{*}$ such that $\pi(\bar{a})\in U$ and $B\subset V$. Then $\pi^{-1}(U)$ and $\pi^{-1}(V)$ are open sets in $G_{\infty}$ which satisfy $\pi^{-1}(\pi(\bar{a}))\subset \pi^{-1}(U)$ and $\pi^{-1}(B)\subset \pi^{-1}(V)$. Also, $\pi^{-1}(\pi(\bar{a}))=\{\bar{a}\}$. Since $\bar{a}\in \pi^{-1}(U)$ and $\pi^{-1}(U)$ is open, there exists a basic open $g_{j}^{-1}(u_{j})$ such that $\bar{a}\in g_{j}^{-1}(u_{j})\subset \pi^{-1}(U)$. This implies that $a_{j}=u_{j}$ and thus $g_{j}^{-1}(a_{j})\subset \pi^{-1}(U)$. By definition of $r_{j}$, we have that $(a_{j},b_{j}^{j})\in r_{j}$. Consider the thread $b_{j}=(b_{1}^{j},b_{2}^{j},b_{3}^{j},..., b_{j-1}^{j},b_{j}^{j},...)$. Then $b_{j}\in A$ and thus, $b_{j}\in \pi^{-1}(B)\subset \pi^{-1}(V)$. Since $\pi^{-1}(V)$ is open, for $b_{j}$ there exists a basic open $g_{i}^{-1}(b_{i}^{j})$ such that $b_{j}\in g_{i}^{-1}(b_{i}^{j})\subset \pi^{-1}(V)$. If $\bar{z}\in g_{i+1}^{-1}(b_{i+1}^{j})$, then $z_{i+1}=b_{i+1}^{j}$ this implies that $z_{i}=g_{i}^{i+1}(b_{i+1}^{j})=b_{i}^{j}$. Hence, $\bar{z}\in g_{i}^{-1}(b_{i}^{j})$ and $g_{i+1}^{-1}(b_{i+1}^{j})\subset g_{i}^{-1}(b_{i}^{j})$. Without losing generality, we suppose that $i>j$. Remember that $L_{i+1}=L_{i}+i$, then $L_{i}+2\leq j+L_{i}+1\leq L_{i+1}$. Applying the row 7 of the bonding function $g_{i}^{i+1}$, we get that $g_{i}^{i+1}(c^{2}_{i+1,j+(L_{i}+1)})=b_{i}^{j}$. Now, we consider the thread 
\begin{equation*}
	\bar{c}=\bigl(b_{1}^{j},b_{2}^{j},b_{3}^{j},..., b_{j}^{j},...,b_{i}^{j}, c^{2}_{i+1,j+(L_{i}+1)},c^{2}_{i+2,j+(L_{i}+1)+1},c^{2}_{i+3,j+(L_{i}+1)+2},...\bigr).
\end{equation*}
Then $\bar{c}\in g_{i}^{-1}(b_{i}^{j})$ and thus, $\bar{c}\in \pi^{-1}(V)$. Define the thread $\bar{d}$ as 
\begin{equation*}
	\bar{d}=\bigl(a_{1},...,g_{i-1}^{i+1}(c^{1}_{i+1,j+(L_{i}+1)}),g_{i}^{i+1}(c^{1}_{i+1,j+(L_{i}+1)}), c^{1}_{i+1,j+(L_{i}+1)},c^{1}_{i+2,j+(L_{i}+1)+1},...\bigr).
\end{equation*}
By definition of $r_{i+k}$ with $k\in \N$, we have that $\bigl(c^{2}_{i+k,j+(L_{i}+k)},c^{1}_{i+k,j+(L_{i}+k)}\bigr)\in r_{i+k}$. Then $\bigl(\bar{d},\bar{c}\bigr)\in r$ and therefore $\bar{d}\in \pi^{-1}(V)$. Since $i>j$, we get that $L_{i}+1\leq j+L_{i}+1\leq L_{i}+i=L_{i+1}$ and applying the row 6 of the bonding function $g_{i}^{i+1}$, it result that $g_{i}^{i+1}\bigl(c^{1}_{i+1,j+(L_{i}+1)}\bigr)=d_{i}^{j}$. If $j=i-1$, by row 9 of the function $g_{i-1}^{i}$, we have that $g_{i-1}^{i}\bigl(d_{i}^{j}\bigr)=a_{i-1}=a_{j}$. In other case ($j<i-1$), applying the row 8 of the function $g_{i-1}^{i}$, $g_{i-1}^{i}\bigl(d_{i}^{j}\bigr)=d_{i-1}^{i-2+j}$. Considering now the function $g_{i-2}^{i-1}$, we have two cases again. If $j=i-2$, then $g_{i-2}^{i-1}\bigl(d_{i-1}^{i-2+j}\bigr)=a_{i-2}=a_{j}$ and in the case $j<i-2$, $g_{i-2}^{i-1}\bigl(d_{i-1}^{i-2+j}\bigr)=d_{i-2}^{i-3+j}$. Since $i>j>1$, there exists $0<l<i$ such that $j=i-l$. If we continue with this process until $j=i-l$, we will get that $g_{i-l}^{i-l+1}\bigl(d_{i-l+1}^{i-l+j}\bigr)=a_{i-l}=a_{j}$. Therefore $g_{j}^{i}\bigl(d_{i}^{j}\bigr)=a_{j}$. Consequently, $\bar{d}\in g_{j}^{-1}(a_{j})$ and thus $\bar{d}\in \pi^{-1}(U)$. So, we get that $\pi^{-1}(U)\cap \pi^{-1}(V)\neq \emptyset$. This mean that there exists $\bar{z}\in G_{\infty}$ such that $\pi(\bar{z})\in U$ y $\pi(\bar{z})\in V$. Hence, $\pi(\bar{z})\in U\cap V$and this implies that $U\cap V\neq \emptyset$. Therefore, if $U$ and $V$ are open sets in $G^{*}$ such that $\pi(\bar{a})\in U$ and $B\subset V$, then $U\cap V\neq \emptyset$. We conclude that $G^{*}$ is not regular. 
\end{proof}

Evidently, there is a net version of Definition~\ref{GCellStruc}, as was observed in \cite{DT2018} for cell structures. Using this generalized definition one may check that uniform spaces $(X,\mathcal{U})$  have a  nontrivial g-cell structure $\{(G_u,u)\}_{u\in\mathcal{U}}$ where $G_u=X$ and $g_u^u=I_X$ for all $u,v\in\mathcal{U}$. It would be interesting to determine what is the class of topological spaces admitting a g-cell structure.

\bibliography{mybib}{}
\bibliographystyle{plain}
\end{document}